\documentclass[11pt]{amsart}
\usepackage[margin=1in]{geometry}
\usepackage{hyperref}
\usepackage{xcolor}
\hypersetup{
    colorlinks,
    linkcolor={red!50!black},
    citecolor={blue!50!black},
    urlcolor={blue!80!black}
}

\usepackage{amssymb}
\usepackage{mathtools}
\usepackage{amsthm}
\usepackage{enumerate}
\usepackage{bbm}

\theoremstyle{plain}
\newtheorem{theorem}{Theorem}[section]

\newtheorem{lemma}[theorem]{Lemma}
\newtheorem{corollary}[theorem]{Corollary}
\newtheorem{conjecture}[theorem]{Conjecture}
\theoremstyle{definition}
\newtheorem{definition}[theorem]{Definition}

\newcommand{\ih}{\hat{\imath}}
\newcommand{\jh}{\hat{\jmath}}

\DeclareMathOperator{\sgn}{sgn}
\DeclareMathOperator{\Spec}{Spec}

\let\latexcirc=\circ
\newcommand{\ccirc}{\mathbin{\mathchoice
  {\xcirc\scriptstyle}
  {\xcirc\scriptstyle}
  {\xcirc\scriptscriptstyle}
  {\xcirc\scriptscriptstyle}
}}
\newcommand{\xcirc}[1]{\vcenter{\hbox{$#1\latexcirc$}}}
\let\circ\ccirc

\begin{document}
\title{Pierce--Birkhoff conjecture is true for splines}
\author[Z.~Lai]{Zehua~Lai}
\address{Department of Mathematics, University of Texas, Austin, TX 78712}
\email{zehua.lai@austin.utexas.edu}
\author[L.-H.~Lim]{Lek-Heng~Lim}
\address{Computational and Applied Mathematics Initiative, Department of Statistics,
University of Chicago, Chicago, IL 60637}
\email{lekheng@uchicago.edu}

\begin{abstract}
We prove the Pierce--Birkhoff conjecture for splines, i.e.,  continuous piecewise polynomials of degree $d$ in $n$ variables on a hyperplane partition of $\mathbb{R}^n$, can be written as a finite lattice combination of polynomials. We will provide a purely existential proof, followed by a more in-depth analysis that yields effective bounds.
\end{abstract}
\maketitle

\section{Introduction}\label{sec:intro}

The Pierce--Birkhoff conjecture states that every continuous piecewise polynomial function on a semialgebraic partition of $\mathbb{R}^n$ is a finite lattice combination of polynomials, i.e., the functions generated by addition, multiplication, and maximization of polynomial functions on $\mathbb{R}^n$. To be precise, a continuous function $f : \mathbb{R}^n \to \mathbb{R}$ that takes a form
\begin{equation}\label{eq:def}
f = g_1 \mathbbm{1}_{\Pi_1} + \dots +  g_r \mathbbm{1}_{\Pi_r}
\end{equation}
where $g_1,\dots,g_r \in \mathbb{R}[x_1,\dots,x_n]$,
\begin{equation}\label{eq:part}
\mathbb{R}^n = \Pi_1 \cup \dots \cup \Pi_r
\end{equation}
a partition, and $\mathbbm{1}_{\Pi}$ the indicator function, is a \emph{spline} of degree $d \coloneqq \max \{ \deg g_1, \dots, \deg g_r \}$ if \eqref{eq:part} is a hyperplane partition, i.e., $\Pi_1,\dots,\Pi_r$ are polytopes. This is the near universal definition of splines in both pure \cite{Billera, Billera1989, Plaumann2012} and applied mathematics \cite{Chui, BoxSpline, Wang2002}. It is equivalent to requiring that \eqref{eq:part} be a triangulation, i.e., $\Pi_1,\dots,\Pi_r$ are simplices, since any hyperplane partition may be refined into a triangulation and vice versa. It also includes all smoother $C^k$-variants as special cases, i.e., the results that we prove in this article for will hold for all $C^k$-splines, $k = 0,1,2,\dots.$

There is however a generalization called \emph{semialgebraic splines} \cite{dipasquale2020, dipasquale2017, shek} that allows for $\Pi_1,\dots,\Pi_r$ to be  semialgebraic sets. In this terminology, the Pierce--Birkhoff conjecture states:
\begin{conjecture}[Pierce--Birkhoff]\label{conj:PB}
Any semialgebraic spline $f : \mathbb{R}^n \to \mathbb{R}$ can be expressed as
\begin{equation}\label{eq:PB}
f=\max_{i=1,\dots,p}\min_{j=1,\dots,p'} f_{ij}
\end{equation}
with finitely many polynomials $f_{ij} \in \mathbb{R}[x_1,\dots,x_n]$, $i =1,\dots,p$, $j=1,\dots,p'$.
\end{conjecture}
Despite its apparent simplicity and longevity (first in-print appearance in 1962), the conjecture has resisted all attempts so far. Aside from the case $d = 1$ and the case $n \le 2$, the Pierce--Birkhoff conjecture is not known to hold for any other values of $d$ and $n$. We will prove that it holds for splines, i.e., when the partition \eqref{eq:part} is a hyperplane partition, for all  $d$ and $n$.

\subsection*{Brief history and prior attempts} The Pierce--Birkhoff conjecture, as noted in \cite{Madden2011}, was first formally formulated by Henriksen and Isbell in  \cite{HI}, with roots in Birkhoff and Pierce's earlier work \cite{BP} in 1956. It has been shown to be true for $n = 1$ and $2$ in \cite{Mahe1984} and for $d= 1$ in \cite{Ovchinnikov2002}. These results have spawned many variations: For $n  = 1$ it has been extended to nonsingular curves \cite{Madden1989} and semialgebraic curves \cite{PB4}. For $n =2$ it has been extended to arbitrary ordered fields \cite{Delzell89}, nonsingular two-dimensional real affine varieties \cite{Wagner10}, regular local rings of dimension two \cite{Lucas12}; it has also been restricted to the hypercube \cite{Lapenta15} and generalized to signomials \cite{Delzell10}. Nevertheless we emphasize that none of these variations go beyond $n = 2$. There are only limited partial results for the $n = 3$ case \cite{PB3, Lucas15};  the latter work \cite{Lucas15} is a partial result for the $n = 3$ case of the connectedness conjecture, which is known to imply the Pierce--Birkhoff conjecture \cite{Lucas09}.

\subsection*{Our contributions} We will show that Conjecture~\ref{conj:PB} holds for any hyperplane partition. We first provide a short existential proof that uses the language of real spectrum \cite{bochnak2013real} and the notion of separating ideals \cite{Madden1989}.  We then give a more careful construction to deduce explicit bounds for $p, p'$ and the degrees of $f_{ij}$'s in terms of $r$, $d$, and $n$.

\section{Splines and semialgebraic splines}\label{sec:splines}

We begin with a more formal description of partitions and splines. For any $b \in \mathbb{N}$, we define the set of all $3^b$ ternary bits by
\begin{equation}\label{eq:Phi}
\Theta_b \coloneqq \bigl\{ \theta: \{1, \dots, b\} \to \{1, 0, -1\} \bigr\}.
\end{equation}
\begin{definition}[Partitions]\label{def:par}
Let $P \coloneqq \{\pi_1, \dots, \pi_b\} \subseteq \mathbb{R}[x_1, \dots, x_n]$ be a set of nonconstant polynomials of degrees not more than $d'$. For any $\theta \in \Theta_b$, we define
\[
\Pi_{\theta} \coloneqq \{x \in\mathbb{R}^n : \sgn(\pi_i(x)) = \theta(i), \; i = 1,\dots,b\}.
\]
Then $\{\Pi_{\theta} : \theta \in \Theta_b \}$ is called a \emph{semialgebraic partition} of $\mathbb{R}^n$ of degree $d'$ induced by $\pi_1,\dots,\pi_b$ \cite{bochnak2013real}. If $d' =1$, i.e., $\pi_1,\dots,\pi_b$ are all degree-one polynomials, then it is called a \emph{hyperplane partition} \cite{Aguiar}. 
We will call each $\Pi_\theta$ a \emph{piece} of the partition. In a semialgebraic partition, any piece $\Pi_\theta$ has a well-defined dimension \cite{bochnak2013real}.
\end{definition}
By far most common and classical partition is a triangulation, i.e., a subdivision of $\mathbb{R}^n$ into a finite simplicial complex \cite{hatcher} or one may also use convex polytopes in place of simplices \cite{Chui, BoxSpline, Wang2002}. The important point to note is that any such ``piecewise linear'' partitions may be refined into a hyperplane partition as in Definition~\ref{def:par}, even if they are not at first defined in this manner.

A partition in the sense  of Definition~\ref{def:par} uniquely determines a partition in the sense of \eqref{eq:part} with $r = 3^b$, i.e.,
\[
\mathbb{R}^n = \bigcup_{\theta \in \Theta_b} \Pi_\theta = \Pi_1 \cup \dots \cup \Pi_r.
\]
Whenever there is no need to concern ourselves with the defining polynomials $\pi_1,\dots,\pi_b$, we will label a partition simply as $\mathbb{R}^n = \Pi_1 \cup \dots \cup \Pi_r$.
\begin{definition}[Chambers]\label{def:cham}
Let  $\{\Pi_{\theta} : \theta \in \Theta_b \}$ be a semialgebraic partition of $\mathbb{R}^n$. A piece of full dimension $n$ is called a \emph{chamber} \cite{Aguiar} if $\{\Pi_{\theta} : \theta \in \Theta_b \}$ is a hyperplane partition and more generally a \emph{semialgebraic chamber}.
\end{definition}

We may now formally define splines and semialgebraic splines.
\begin{definition}[Splines]\label{def:piecewise}
Let $\{\Pi_{\theta} : \theta \in \Theta_b \}$  be semialgebraic partition of $\mathbb{R}^n$ of degree $d'$. A continuous function $f: \mathbb{R}^n \to \mathbb{R}$ is a \emph{semialgebraic spline} of degree $d$  if $f$ restricts to  a polynomial  $g_\theta  \in \mathbb{R}[x_1, \dots, x_n]$ of degree not more than $d$ on $\Pi_\theta$ whenever $\Pi_\theta \ne \varnothing$.
A semialgebraic spline on a hyperplane partition, i.e., $d' = 1$, is called a \emph{spline}.
\end{definition}
To be clear, we have written
\[
d \coloneqq \max \{ \deg g_\theta : \theta \in \Theta_b \},  \qquad  d' \coloneqq \max \{ \deg \pi_i : i = 1, \dots, b\}
\]
to distinguish the two different degrees tied to a semialgebraic spline.

We will be invoking this simple observation several times later: Given a semialgebraic partition $\mathbb{R}^n = \Pi_1 \cup \dots \cup \Pi_r$, assume without loss of generality that $\Pi_1,\dots,\Pi_p$, $p \le r$, are all the chambers in this partition. Let $V_i$ be the (Euclidean) closure of $\Pi_i$, $i = 1,\dots,p$. Then $\mathbb{R}^n = V_1 \cup \dots \cup V_p$. By continuity, the restriction of $f$ to $\Pi_i$ completely determines $f$ on $V_i$, $i=1,\dots,p$.  A consequence is that $g_{p+1},\dots,g_r$, i.e., restrictions of $f$ to non-chamber pieces, do not play a role in our results.

\begin{definition}[Max-definable function]\label{def:maxdef}
A \emph{max-definable functions} in the variables $x_1,\dots,x_n$ is a function $f : \mathbb{R}^n \to \mathbb{R}$  generated by indeterminates $x_1, \dots, x_n$ and constants $c\in \mathbb{R}$ under four operations: addition $(x,y) \mapsto x+y$, multiplication $(x,y) \mapsto x \cdot y$, maximization $(x,y) \mapsto \max(x,y)$, and scalar multiplication $(c, x) \mapsto cx$.
\end{definition}
Note that minimization comes for free in Definition~\ref{def:maxdef} as $\min(x,y) \coloneqq -\max(-x,-y)$. The identity
\[
x\max(y,0) = \max\bigl(\min(xy, x^2y+y), \min(0, -x^2y-y)\bigr)
\]
allows us to exchange the order of multiplication and maximization using this rule. Consequently any max-definable functions can be written in the form $\max_{i=1,\dots,p}\min_{j=1,\dots,p'} g_{ij}$ as in the statement of Pierce--Birkhoff conjecture \cite{HI}. We will prove that any spline is max-definable, i.e., Conjecture~\ref{conj:PB} is true for any $d$ and $n$ when $d' = 1$.

\section{Existential proof}\label{sec:noneff}

The proof in this section is based on the notion of separating ideals \cite{Madden1989}. We will present a simplified treatment defined for two semialgebraic sets $U, V$, which is all we need, after recalling some real algebraic geometry \cite{bochnak2013real, Madden1989}.

The real spectrum of a ring $R$, denoted $\Spec_{\mathbb{R}}(R)$, is a set of pairs $(\mathfrak{p}_\alpha, >_\alpha) \eqqcolon \alpha$, where $\mathfrak{p}_\alpha$ is a prime ideal and $>_\alpha$ is a total order on $R/\mathfrak{p}_\alpha$ \cite[Proposition~7.1.2]{bochnak2013real}.  For any $f \in R$, let
\begin{equation}\label{eq:extend}
f(\alpha) \coloneqq f + \mathfrak{p}_\alpha \in R/\mathfrak{p}_\alpha.
\end{equation}
If $\mathfrak{p}_\alpha \subseteq \mathfrak{p}_\beta$ and the natural map $R/\mathfrak{p}_\alpha \to R/\mathfrak{p}_\beta$ preserves order, we write $\alpha \subseteq \beta$ and say that $\beta$ \emph{specializes} $\alpha$. The topology on $\Spec_{\mathbb{R}}(R)$ is generated by basic open sets of the form
\[
\widetilde{B}(f) = \{\alpha \in \Spec_{\mathbb{R}}(R)  : f(\alpha) >_\alpha 0\}
\]
over all $f \in R$. Under this topology, $\beta$ specializes $\alpha$ if and only if $\beta$ is in the closure of $\{\alpha\}$  \cite[Proposition~7.1.18]{bochnak2013real}.

Henceforth, we will limit our attention to $R \coloneqq \mathbb{R}[x_1, \dots, x_n]$.  We identify $a \in \mathbb{R}^n$ with its vanishing ideal $\mathfrak{m}_a$ and let the unique order on  $\mathfrak{m}_a$ be given by $f < g$ if and only if $f(a) < g(a)$. This defines an injection $\mathbb{R}^n \to \Spec_{\mathbb{R}}(R)$ and we may identify $\mathbb{R}^n$ with its image in $\Spec_{\mathbb{R}}(R)$. Given a semialgebraic set $V$ defined by a Boolean combination of sets of the form
\[
B(f) = \{a \in \mathbb{R}^n  : f(a) > 0\}
\]
over some collection of  $f\in R$, the same Boolean combination of $\widetilde{B}(f)$ defines a set $\widetilde{V}$ in $ \Spec_{\mathbb{R}}(R) $. We also set $\widetilde{\varnothing} = \varnothing$. Clearly $\widetilde{V} \cap \mathbb{R}^n = V$. In fact $\widetilde{V}$ is uniquely determined by $V$ and does not depend on the choice of Boolean expressions used to define $V$ \cite[Proposition 7.2.2]{bochnak2013real}.

For any $f \in R$, the polynomial function $f : \mathbb{R}^n \to \mathbb{R}$ extends to a function $f : \Spec_{\mathbb{R}}(R) \to \mathbb{R}$ as in \eqref{eq:extend}. Note that $f \ge 0$ on $V$ if and only if $f \ge 0$ on $\widetilde{V}$. Hence a semialgebraic partition of $\mathbb{R}^n$ gives a partition of $\Spec_{\mathbb{R}}(R)$. A semialgebraic spline $f$ may then be extended to $\Spec_{\mathbb{R}}(R)$ in a natural way: If $f$ is as in \eqref{eq:def}, then it can be extended to $\Spec_{\mathbb{R}}(R)$ as
\[
f = g_1 \mathbbm{1}_{\widetilde{\Pi}_1} + \dots +  g_r \mathbbm{1}_{\widetilde{\Pi}_r},
\]
where $g_i : \Spec_{\mathbb{R}}(R) \to \mathbb{R}$ extends the polynomial function $g_i$, $i =1,\dots,r$. We will write $f_\alpha \in R$ for a polynomial that agrees with $f$ at $\alpha \in  \Spec_{\mathbb{R}}(R)$. The choice of $f_{\alpha}$ will not affect any result in this article. Lastly, by virtue of the way they are defined, addition, multiplication, and maximum of two semialgebraic splines over $\mathbb{R}^n$ extend naturally over $\Spec_{\mathbb{R}}(R)$.

We will define separating ideal as in \cite{Madden1989} and introduce a slight generalization.
\begin{definition}[Separating ideal]
Let $\alpha, \beta \in \Spec_{\mathbb{R}}(R)$. We write $\langle \alpha, \beta \rangle$ for the ideal generated by all $f \in R$ such that $f(\alpha) \ge 0$ and $f(\beta) \le 0$.  This is called the separating ideal of $\alpha$ and $\beta$. Let $U, V$ be semialgebraic sets. Likewise we will write $\langle U, V \rangle$ for the ideal of $R$ generated by all $f \in R$ such that $f \ge 0$ on $U$ and $f \le 0$ on $V$, and call it the separating ideal of $U$ and $V$.
\end{definition}

Since $f \ge 0$ on $U$ if and only if $f \ge 0$ on $\widetilde{U}$ for any $f \in R$, the latter definition extends unambiguously to $\langle \widetilde{U}, \widetilde{V} \rangle$. It is clear that if $\alpha \in \widetilde{U}$ and $\beta \in \widetilde{V}$, then $\langle U, V \rangle \subseteq \langle \alpha, \beta \rangle$. For easy reference, we reproduce the main result of \cite{Madden1989}.
\begin{theorem}[Madden, 1989]\label{thm:madden}
A semialgebraic spline $f$ is max-definable if and only if for any $\alpha, \beta \in \Spec_{\mathbb{R}}(R)$, $f_\alpha - f_\beta \in \langle \alpha, \beta \rangle$.
\end{theorem}
This alternative formulation of Pierce--Birkhoff conjecture has come to be known as the local Pierce--Birkhoff conjecture \cite{Lucas15}. While it sheds light on the original conjecture, it does not appear to be any easier. Neither the separating ideal $\langle \alpha, \beta \rangle$ nor functions of the form $f_\alpha - f_\beta$ are well understood. For instance, it is known that the separating ideal of disjoint semialgebraic sets $\langle U, V \rangle$ can be zero: Take $U = \{(x, y)\in \mathbb{R}^2 : xy \ge 1, \; x > 0\}$ and $V  = \mathbb{R}^2 \setminus U$, then $\langle U, V \rangle = \{0\}$. There are also examples of dimension three \cite{B88}. But there is no known example of the form $\langle \alpha, \beta \rangle$, which would be a potential counterexample for the local Pierce--Birkhoff conjecture.

Nevertheless, Theorem~\ref{thm:madden} is useful to us in the form of the following corollary.
\begin{corollary}\label{cor:separating}
Let $f$ be a semialgebraic spline over  a semialgebraic partition $\mathbb{R}^n = \Pi_1 \cup \dots \cup \Pi_r$ with chambers $\Pi_1,\dots,\Pi_p$,  $p \le r$. Let $g_i \in \mathbb{R}[x_1, \dots, x_n]$ be the restriction of $f$ to $\Pi_i$, and $V_i $ be the closure of $\Pi_i$, $i =1,\dots,p$. If $g_i - g_j \in \langle V_i, V_j \rangle$, then $f$ is max-definable. 
\end{corollary}
\begin{proof}
Given any $\alpha, \beta \in \Spec_{\mathbb{R}}(R)$, we need to verify the condition in Theorem~\ref{thm:madden}. Since $V_1, \dots, V_p$ cover $\mathbb{R}^n$, $\widetilde{V}_1, \dots, \widetilde{V}_p$ cover $\Spec_{\mathbb{R}}(R)$. For any $\alpha, \beta$, we can find $\alpha \in \widetilde{V_i}, \beta \in \widetilde{V_j}$, then $f_\alpha - f_\beta = g_i - g_j \in \langle V_i, V_j\rangle \subseteq \langle \alpha, \beta \rangle$. Hence $f$ is max-definable. 
\end{proof}
While Corollary~\ref{cor:separating} has not been explicitly stated anywhere, the proof of Pierce--Birkhoff conjecture for $n =2$ \cite{Mahe1984} essentially uses this simplified version. By a further refinement if necessary, we may assume that the partition $\mathbb{R}^n = \Pi_1 \cup \dots \cup \Pi_r$ gives a cylindrical algebraic decomposition \cite{bochnak2013real}. For $n = 2$, the intersection $V_i \cap V_j$ is either empty, a point, or a curve. It is relatively straightforward to study how two sets in the plane can intersect in these three cases and arrive at the Pierce--Birkhoff conjecture. However, for $n \ge 3$, the only easy case is when $V_i \cap V_j$ is a hypersurface with the separating ideal a principal ideal. When codimension exceeds one, the separating ideal becomes far more complicated. 

Fortunately, if we limit ourselves to splines as in Definition~\ref{def:piecewise}, then the intersections are all convex polytopes, which brings the condition in Corollary~\ref{cor:separating} within reach of standard convex geometry.
\begin{theorem}[Pierce--Birkhoff is true for splines]\label{thm:PBlinear1}
Let $f$ be a spline. Then $f$ is max-definable.
\end{theorem}
\begin{proof}
Let $f$ be a spline as in Definition~\ref{def:piecewise}. We follow the notations in Corollary~\ref{cor:separating}. Let chambers in the partition be $\Pi_1,\dots,\Pi_p$, and $V_i $ be the closure of $\Pi_i$, $i =1,\dots,p$. For any two $V_i, V_j$. The intersection $V_i \cap V_j$ is a closed convex polytope. Let $\mathfrak{a}$ be the vanishing ideal of the affine span of $V_i \cap V_j$. We claim that $\langle V_i, V_j\rangle = \mathfrak{a}$. 

We first consider the case $V_i \cap V_j \neq \varnothing$. One direction is clear: If $h \in \langle V_i, V_j\rangle$, then $h = 0$ in the closed convex set $V_i \cap V_j$, so $h = 0$ in its affine span, thus $\langle V_i, V_j\rangle \subseteq \mathfrak{a}$. 

Conversely, choose a point $x$ in the relative interior of $V_i \cap V_j$. Let $C_i$ and $C_j$ be the tangent cones of $V_i$ and $V_j$ at $x$. Then $C_i \cap C_j$ is the affine span of $V_i \cap V_j$ and it is the largest affine subset that $C_i$ or $C_j$ contains. Consider the quotients $C_i/(C_i\cap C_j)$ and $C_j/(C_i\cap C_j)$, or, alternatively, the projections of $C_i$ and $C_j$ to the orthogonal complement of $C_i\cap C_j$. These are two closed disjoint pointed polyhedral cones and thus there is a hyperplane $\{x \in \mathbb{R}^n : h(x) = 0 \}$ that strictly separates the two cones \cite{Klee1955}. By definition, the degree-one polynomial $h \in \langle V_i, V_j\rangle$. For any degree-one polynomial $h' \in \mathfrak{a}$, the hyperplane $\{ x  \in \mathbb{R}^n : h'(x) = 0 \}$ contains $C_i \cap C_j$. So $h + \varepsilon h'$ will still separate the two cones for sufficiently small $\varepsilon > 0$. Therefore $h + \varepsilon h' \in \langle V_i, V_j\rangle$ and thus $h' \in \langle V_i, V_j\rangle$. Hence $\mathfrak{a} \subseteq \langle V_i, V_j\rangle$. In conclusion, $\mathfrak{a} = \langle V_i, V_j\rangle$.

If $V_i \cap V_j = \varnothing$, there is also a strictly separating hyperplane \cite[Theorem~1.1.10]{stoer2012}. Hence we can find $h$ so that $h + \varepsilon$ separates $V_i$ and $V_j$ for all sufficiently small $\varepsilon > 0$. Thus $1 \in \langle V_i, V_j\rangle$, i.e., $\langle V_i, V_j\rangle = R$. In this case the affine span of $V_i \cap V_j$ is the empty set and we also get $\langle V_i, V_j\rangle = \mathfrak{a}$.

By its definition, a spline $f$ has $f\rvert_{V_i} = g_i$, $i=1,\dots,p$. Since $g_i - g_j = 0$ is in $V_i \cap V_j$, $g_i-g_j \in \mathfrak{a} = \langle V_i, V_j\rangle$.  By Corollary~\ref{cor:separating}, $f$ is max-definable.
\end{proof}

We will see in Theorem~\ref{thm:eff} that Theorem~\ref{thm:PBlinear1} can be made much more informative, but at some cost to clarity. The abstract approach above, freed of the need to keep track of polynomials, is more likely to generalize to other semialgebraic splines.

\section{Effective bounds}\label{sec:eff}

In this section, we will derive explicit bounds for Theorem~\ref{thm:PBlinear1}, essentially by making effective the abstract constructions in \cite{Madden1989}.
\begin{lemma}\label{lemma:eff1}
Let $f$ be a semialgebraic spline over  a semialgebraic partition $\mathbb{R}^n = \Pi_1 \cup \dots \cup \Pi_r$ with chambers $\Pi_1,\dots,\Pi_p$,  $p \le r$. Let $g_i \in \mathbb{R}[x_1, \dots, x_n]$ be the restriction of $f$ to $\Pi_i$, and $V_i $ be the closure of $\Pi_i$, $i =1,\dots,p$. If for $ i, j = 1,\dots, p$, there exists $h_{ij} \in \mathbb{R}[x_1, \dots, x_n]$ such that $h_{ij} \le g_j-g_i$ on $V_j$ and $h_{ij} \ge 0 $ on $V_i$, then $f$ is max-definable with
\[
f = \max_{i = 1,\dots,p} \min_{j = 1,\dots,p} f_{ij}
\]
for some $f_{ij} \in \mathbb{R}[x_1, \dots, x_n]$ where
\[
\max \{ \deg f_{ij} : i, j = 1,\dots,p\} \le \max \{\deg g_i, \deg h_{ij} : i, j = 1,\dots, p\}.
\]
\end{lemma}
\begin{proof}
For any $i, j = 1,\dots, p$, let
\[
f_{ij} \coloneqq h_{ij}+g_i.
\]
So the assumption becomes $f_{ij} \le g_j$  on $V_j$ and $f_{ij} \ge g_i $ on $V_i$. For $i = j$, the two conditions on $h_{ij}$ forces $h_{ii} = 0$, so $f_{ii} = g_i$. Set
\[
f_0 \coloneqq \max_{i = 1,\dots,p} \min_{j = 1, \dots, p}f_{ij}.
\]
On $V_k$, for any $i \neq k$,
\[
\min_{j = 1,\dots,p} f_{ij} \le f_{jk} \le g_k,
\]
and for $i = k$,
\[
\min_{j = 1,\dots,p} f_{ij} = f_{ii} = g_k.
\]
So
\[
f_0 = \max_{i = 1,\dots,p} \min_{j = 1,\dots,p} f_{ij} = g_k
\]
on $V_k$. Since $f\rvert_{V_k} = g_k$, $k=1,\dots,p$, we have $f_0 = f$ and the degree bound follows.
\end{proof}

We now establish our final result.
\begin{theorem}[Effective Pierce--Birkhoff for splines]\label{thm:eff}
Let $f$ be a spline over a hyperplane partition $\mathbb{R}^n = \Pi_1 \cup \dots \cup \Pi_r$ induced by degree-one polynomials $P = \{\pi_1, \dots, \pi_b\}$. Let $\Pi_1,\dots,\Pi_m$,  $m \le r$, be the chambers, $g_i \in \mathbb{R}[x_1, \dots, x_n]$ be the restriction of $f$ to $\Pi_i$, and $V_i $ be the closure of $\Pi_i$, $i =1,\dots,m$. If  $d  = \max \{ \deg g_i : i = 1,\dots,m\}$, then $f$ is max-definable with
\[
f = \max_{i = 1,\dots,p} \min_{j = 1,\dots,p} f_{ij}
\]
for some $f_{ij} \in \mathbb{R}[x_1, \dots, x_n]$, $\deg f_{ij} \le 2d+1$, and
\[
p = O\bigl( b^{2n^2}d^n \bigr)
\]
for any fixed $n \in \mathbb{N}$.
\end{theorem}
\begin{proof}
Let $i,j \in \{1,\dots,m\}$. 
Let $\mathfrak{a}$ be the vanishing ideal of the affine span of $V_i \cap V_j$. As we saw in the proof of Theorem~\ref{thm:PBlinear1}, $\mathfrak{a} = \langle V_i, V_j\rangle$. When $V_i \cap V_j \neq \varnothing$, $\mathfrak{a}$ is an ideal generated by degree-one polynomials and $g_i-g_j$ can then be written as
\begin{equation}\label{eq:gigi}
g_i - g_j = \sum_{k=1}^q \pi_{ijk} \rho_{ijk}
\end{equation}
where $\deg \pi_{ijk} = 1$, $\deg \rho_{ijk} \le \deg (g_i - g_j) -1$, and $q \coloneqq q_{ij}$ is understood to depend on $i,j$ but denoted as $q$ to avoid clutter (we will later see that $q_{ij} \le n+1$ irrespective of $i,j$). Furthermore, the construction in the proof of Theorem~\ref{thm:PBlinear1} shows that we may choose $\pi_{ijk}$ so that they separate $V_i$ and $V_j$. When $V_i \cap V_j = \varnothing$, $g_i-g_j$ can similarly be expressed as in \eqref{eq:gigi}, where $\pi_{ijk}$ are degree-one separating polynomials and $\deg \rho_{ijk} \le \deg (g_i - g_j)$.  Now we add these polynomials to $P$ to obtain
\begin{equation}\label{eq:P'}
P' = P  \cup \{ g_i - g_j, \; \rho_{ijk}, \; \rho_{ijk}-1 : i,j = 1,\dots, m, \; k = 1,\dots,q \}.
\end{equation}
Then $P'$ induces a finer semialgebraic (but not necessarily hyperplane) partition $\Pi'$. Let
\[
\{V_{i \ih } : i = 1,\dots,m, \; \ih  = 1, \dots, p_i\}
\]
be an exhaustive list of the closures of chambers in $\Pi'$ so that we have
\[
\mathbb{R}^n = \bigcup_{i =1}^m V_i \quad\text{and}\quad
V_i = \bigcup_{\ih =1}^{p_i} V_{i \ih }, \quad i =1,\dots,m.
\]
Let $p \coloneqq  p_1 + \dots + p_m$. This refined partition $\Pi'$ has the property that any $g_i- g_j$ and any $\rho_{ijk}$ and $\rho_{ijk}-1$ is either nonnegative or nonpositive on any $V_{i \ih }$.

Take any pair $V_{i \ih }$ and  $V_{j\jh }$. Suppose $i = j$. Then $f\rvert_{V_{i \ih }} = f\rvert_{V_{j\jh }} = g_i$. So $0$ satisfies the condition in Lemma~\ref{lemma:eff1} that $0 \le g_i-g_i$ on $V_{i \ih }$ and $0 \ge 0 $ on $V_{j\jh }$. Suppose $i \neq j$. If $g_i - g_j \ge 0$ on $V_i$, then $0$ satisfies the condition that $0 \le g_i-g_j$ on $V_{i \ih }$ and $0 \ge 0 $ on $V_{j\jh }$. It remains to consider the case $g_i - g_j \le 0$ on $V_i$. Write $g_i - g_j$ as in \eqref{eq:gigi}. We may choose the sign of $\pi_{ijk}$ and $\rho_{ijk}$ so that $\pi_{ijk} \ge 0$ on $V_{i \ih }$, and $\pi_{ijk} \le 0$ on $V_{j\jh }$. Now define
\[
\sigma_{ijk} \coloneqq \begin{cases}
0 &\rho_{ijk} \le 0  \text{ on } V_{i \ih };\\
\rho_{ijk} &\rho_{ijk} \ge 0  \text{ on } V_{i \ih },\; \rho_{ijk} \ge 0  \text{ on } V_{j\jh };\\
1 &0 \le \rho_{ijk} \le 1  \text{ on } V_{i \ih },\; \rho_{ijk} \le 0  \text{ on } V_{j\jh };\\
\rho_{ijk}^2 &\rho_{ijk} \ge 1  \text{ on } V_{i \ih },\; \rho_{ijk} \le 0  \text{ on } V_{j\jh }.
\end{cases}
\] 
It is straightforward to check that $g_i - g_j \le \sum_{k=1}^q \pi_{ijk} \sigma_{ijk}$ on $V_{i \ih }$,  $\sum_{k=1}^q \pi_{ijk} \sigma_{ijk} \ge 0$ on $V_{j\jh }$, and
\[
\deg \biggl( \sum_{k=1}^q \pi_{ijk}\sigma_{ijk}  \biggr) \le 2\deg (g_i - g_j) +1 \le 2d +1.
\]
Hence by Lemma~\ref{lemma:eff1}, $f$ can be written as
\[
f = \max_{i = 1, \dots, p} \min_{\ih  = 1, \dots, p} f_{i \ih }
\]
where $f_{i \ih } \in \mathbb{R}[x_1, \dots, x_n]$ has the desired degree bound $2d +1$.

To obtain the bound for $p$, we need to count the number of chambers in the semialgebraic partition $\Pi'$. We started from the hyperplane partition $\Pi$ induced by $P$. Since $\lvert P \rvert = b$, the partition $\Pi$ is cut out by $b$ hyperplanes and it can have at most $O(b^n)$ chambers \cite[Chapter~6]{GTM212}. So $m = O(b^n)$. The refined semialgebraic partition $\Pi'$ is induced by $P'$ in \eqref{eq:P'}. There are $m^2 = O(b^{2n})$ polynomials of the form $g_i - g_j$ in $P'$. For each fixed $i, j \in \{1,\dots,m\}$, we may choose $\pi_{ijk}$ in \eqref{eq:gigi}  so that the number of summands $q = q_{ij} \le n+1$. To see this, first suppose $V_i\cap V_j \ne \varnothing$. Let $\dim(V_i\cap V_j) = n_{ij} \le n$ and let $h_1, \dots, h_{n - n_{ij}}$ generate $\langle V_i, V_j \rangle$. Then we may choose $\pi_{ijk}$ to be of the forms $h, h + \varepsilon h_1, \dots, h + \varepsilon h_{n - n_{ij}}$ where $\varepsilon > 0$ is a fixed constant as in the proof of Theorem~\ref{thm:PBlinear1}. Thus $q_{ij} \le n - n_{ij} + 1 \le n+1$. Now suppose $V_i\cap V_j = \varnothing$. Then we may choose $\pi_{ijk}$  to be of the two forms $h, h + \varepsilon$. Thus $q_{ij} \le 2 \le n+1$. Hence there are at most $2m^2 (n+1) = O(b^{2n})$ polynomials of the form $\rho_{ijk}$ or $\rho_{ijk}-1$ in $P'$. All included, we obtain
\[
\lvert P' \rvert \le b + m^2 + 2m^2(n + 1) = O(b) + O(b^{2n}) + O(b^{2n}) = O(b^{2n}).
\]
Since the polynomials in $P'$ have degrees $d' \le 2d$, it follows from \cite[Chapter~7]{basu2003} that $p$, the number of semialgebraic chambers in $\Pi'$, is bounded by
\[
O\bigl((\lvert P'\rvert \cdot d')^n \bigr) = O\bigl((b^{2n} \cdot 2d)^n\bigr) = O\bigl(b^{2n^2}d^n \bigr)
\]
as required.
\end{proof}
A slightly more careful count would yield an explicit expression for the leading constant $c_n$, which depends only on $n$, in the $O\bigl(b^{2n^2}d^n \bigr)$ bound. We see little point in providing it since but the bound for $p$ is already exponential in $n$.

\section{Conclusion}

While our goal in this article is purely mathematical, there is one real-world consequence that is perhaps worth highlighting. By \cite[Theorem~3.8]{lai2024}, it follows from Theorem~\ref{thm:PBlinear1} that every spline is a ReLU-activated transformer.

\bibliographystyle{abbrv}

\end{document}